\documentclass[11pt]{amsart}
\usepackage[a4paper,hmargin=3.5cm,vmargin=4cm]{geometry}
\usepackage{graphicx} 
\usepackage{verbatim,enumerate}
\usepackage[dvipsnames]{xcolor}
\usepackage[colorlinks]{hyperref}
\graphicspath{{./Figures/}}
\usepackage {svg}
\usepackage{wrapfig}
\usepackage{amsrefs}
\usepackage{amssymb, amsmath}

\usepackage{pb-diagram}

\title[Reflection principle for lightlike line segments]{%
Reflection principle for lightlike line segments on maximal surfaces
}

\author[S.~Akamine and H.~Fujino]{
Shintaro Akamine and Hiroki Fujino
}   

\address[Shintaro Akamine]{%
Graduate School of Mathematics, 
Nagoya University, Chikusa-ku, Nagoya 464-8602, Japan
}
\email{s-akamine@math.nagoya-u.ac.jp}

\address[Hiroki Fujino]{%
  Institute for Advanced Research, Graduate School of Mathematics, 
Nagoya University, Chikusa-ku, Nagoya 464-8602, Japan
}
\email{m12040w@math.nagoya-u.ac.jp}

\date{\today}

\thanks{
The first author was partially supported by 
JSPS KAKENHI Grant Number 19K14527, 17H06466 and JSPS/FWF Bilateral Joint Project I3809-N32 ``Geometric Shape Generation'',
and the second author by JSPS KAKENHI Grant Number 19K21022.
}

\usepackage{amsthm}
\theoremstyle{plain}
 \newtheorem{introtheorem}{Theorem}

 \newtheorem{theorem}{Theorem}[section]
 \newtheorem{proposition}[theorem]{Proposition}
 
 \newtheorem{fact}[theorem]{Fact}
 \newtheorem{lemma}[theorem]{Lemma}

 \newtheorem{question}{Question}
\theoremstyle{definition}
 \newtheorem{definition}[theorem]{Definition}
\theoremstyle{remark}
 \newtheorem{remark}[theorem]{Remark}
 \newtheorem*{remark*}{Remark}
\newtheorem{example}[theorem]{Example}
 \newtheorem*{acknowledgement}{Acknowledgement}

\numberwithin{equation}{section}

\renewcommand{\phi}{\varphi}

\newcommand{\Arg}{{\rm Arg}}

\definecolor{Blue}{rgb}{0,0,1}  
\definecolor{Red}{rgb}{1,0,0}  
\begin{document}
\maketitle

\begin{abstract}
As in the case of minimal surfaces in the Euclidean 3-space, the reflection principle for maximal surfaces in the Lorentz-Minkowski 3-space asserts that if a maximal surface has a spacelike line segment $L$, the surface is invariant under the $180^\circ$-rotation with respect to $L$. However, such a reflection property does not hold for lightlike line segments on the boundaries of maximal surfaces in general. 

In this paper, we show some kind of reflection principle for lightlike line segments on the boundaries of maximal surfaces when lightlike line segments are connecting shrinking singularities. As an application, we construct various examples of periodic maximal surfaces with lightlike lines from tessellations of $\mathbb{R}^2$.\end{abstract}

%============================================================INTRODUCTION====
%\tableofcontents

\section{Introduction} \label{sec:1} 
The classical Schwarz reflection principle for harmonic functions yields a symmetry  principle for minimal surfaces in the 3-dimensional Euclidean space $\mathbb{E}^3$: if a minimal surface in $\mathbb{E}^3$ has a straight line segment $L$ on its boundary, then the surface can be extended across $L$ and the extended surface is invariant under the $180^\circ$-rotation  with respect to $L$ (see \cite[p.~289]{DHS}, \cite[p.~140]{N} and \cite[p.~54]{O} for example). 
This principle is directly derived from the Schwarz reflection principle and the fact that each coordinate function of a conformal minimal immersion in $\mathbb{E}^3$ is harmonic. For the same reason, such a reflection principle for lines also holds for maximal surfaces, i.e. spacelike surfaces with vanishing mean curvature, in the 3-dimensional Lorentz-Minkowski space $\mathbb{L}^3$, when the straight line segment is spacelike, see Al\'{i}as-Chaves-Mira \cite[Theorem 3.10]{ACM}. 
As a singular version of this reflection principle, a reflection principle inducing point symmetries for shrinking or conelike singularities was also shown in \cite{FL2,FLS2,LLS,KY} (see also \cite{FLS,FRUYY,IK,K2}).

These regular and singular versions of reflection principles for maximal surfaces are highly depend on the conformal structures of surfaces. However, as another possibility of a line reflection principle, {\it lightlike} line segments can appear on the boundaries of maximal surfaces. Unfortunately, it was shown in the author's previous work \cite{AF} that these boundary lightlike line segments appear as discontinuous boundary behaviors of conformal maximal immersions and hence conformal structures break down on these lines. Therefore, we cannot expect a conventional symmetry principle for lightlike line segments on maximal surfaces in general. In fact, a maximal surface with lightlike line segments along which neither a line symmetry nor planar symmetry holds was given in \cite[Section 4.5]{AF}.
In particular, as far as the authors know, the following problem was still open.

\begin{question} \label{q:1}
 Is there a reflection principle for boundary lightlike line segments of maximal surfaces?
\end{question}

\noindent We can also see a similar question in \cite[p.1095]{KKSY} for boundary lightlike line segments for timelike minimal surfaces in $\mathbb{L}^3$. 
 
Answering Question \ref{q:1}, in the present article we show a reflection principle for a boundary lightlike line segment connecting two shrinking singularities:

\begin{introtheorem}\label{thm:1}
	If a maximal graph $\Sigma$ has a lightlike line segment $L\subset \partial{\Sigma}$ connecting two shrinking singularities as the endpoints of $L$. Then $\Sigma$ can be extended to a surface with zero mean curvature across $L$ which is invariant under the point symmetry at the midpoint of $L$. 
\end{introtheorem}

For a more precise statement, see Theorem \ref{thm:Main}. 
The situation in Theorem \ref{thm:1} is very typical since a large number of maximal surfaces with lightlike line segments have shrinking singularities as their endpoints. See Figure \ref{fig:Periodic} and Example \ref{example:knownperiodic}.

\begin{figure}[htb]
 \begin{center}
  \includegraphics[height=3.5cm]{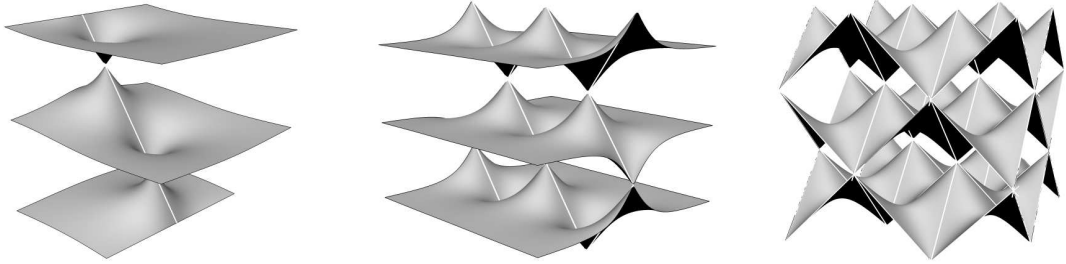} 
\caption{Examples of periodic maximal surfaces with lightlike lines (white lines) connecting shrinking singularities.}\label{fig:Periodic} 
 \end{center}
\end{figure}

Moreover, as an important application of Theorem \ref{thm:1}, we give a construction of new proper periodic maximal surfaces with lightlike line segments as in Figure \ref{fig:hexagon}.
\begin{figure}[h]
 \begin{center}
  \includegraphics[height=3.3cm]{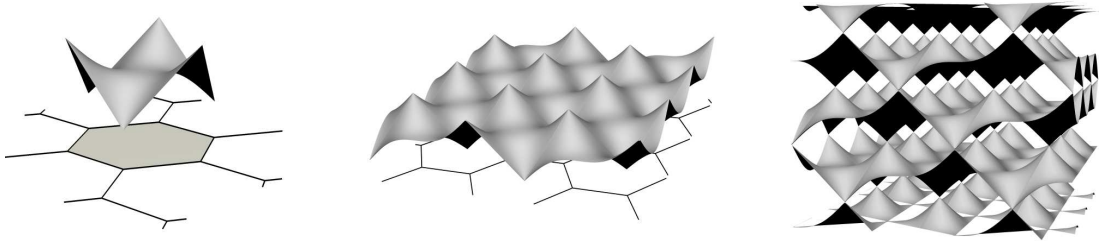} 
\caption{Doubly (center) and triply periodic (right) maximal surfaces generated by the regular hexagonal tiling of $\mathbb{R}^2$.}\label{fig:hexagon} 
 \end{center}
\end{figure}

The organization of this article is as follows: In Section \ref{sec:2.1}, we prepare a necessary condition for existence of extensions of maximal surfaces across lightlike line segments. In Section \ref{sec:2.2}, we explain how lightlike lines and shrinking singularities appear on the boundaries of maximal surfaces from aspects of harmonic function theory.
In Section \ref{sec:2.3}, we give a proof of Theorem \ref{thm:1}. This is achieved by function theoretical ``blow-up'' method of discontinuous boundary points of harmonic mappings, instead of the Schwarz reflection principle, which is a theorem about continuous boundary points of harmonic functions.
 In Section \ref{sec:3}, we give a construction procedure of some periodic maximal surfaces with lightlike line segments by using Theorem \ref{thm:1} and tessellations of $\mathbb{R}^2$ associated with symmetries appearing in Theorem \ref{thm:1}. Finally, we discuss in Section \ref{sec:4} the remaining cases, that is, reflection principles for entire lightlike lines and lightlike half-lines as a future question.

%===============================================================SECTION Main Theorem===========
\section{blow-up of discontinuous boundary points}\label{sec:2}%``open gluing''
In this paper, $\mathbb{L}^3$ denotes the Lorentz-Minkowski 3-space with signature $(+,+,-)$.
Throughout this section, unless otherwise noted, we let $\Sigma={\rm graph}(\psi)\subset \mathbb{L}^3$ be a maximal graph, defined by a smooth function $\psi\colon \Omega\to \mathbb{R}$ over a simply connected bounded Jordan domain $\Omega$ in the $xy$-plane. Here, we assume that $\Sigma$ contains a lightlike line segment in its boundary. In this section, we discuss an extension of $\Sigma$ across the boundary lightlike line segment.

\subsection{A necessary condition.}\label{sec:2.1}
For the existence of such an extension, a necessary condition was introduced in \cite[Definition 3.1]{AF}: Let $I\subset \partial \Omega$ be an open Euclidean segment, and $\tau\in \mathbb{R}^2$ the unit tangent vector to $I$ with the positive direction. Here, the positive direction is determined by the counterclockwise orientation of $\partial \Omega$. We say that $\psi$ {\it tamely degenerates to a future (resp. past)-directed lightlike line segment over $I$} if $\psi$ satisfies
\[
	\frac{\partial \psi}{\partial \tau}(x,y)=1+O({\rm dist}\left((x,y),J)^2\right)\ \ \left({\rm resp.}\ \ 
	\frac{\partial \psi}{\partial \tau}(x,y)=-1+O({\rm dist}\left((x,y),J)^2\right)\right)
\]
as $(x,y)\to J$ for each closed segment $J\subset I$, where $\partial/\partial \tau$ denotes the directional derivative in the $\tau$-direction. We simply say that $\psi$ {\it tamely degenerates to a lightlike line segment} if it tamely degenerates to a future or past-directed lightlike line segment. In this case, it is easily seen that $\partial \Sigma$ contains a lightlike line segment over $I$. Conversely, if $\psi$ extends to a $C^2$-function across $I$ and its graph contains a lightlike line segment $L$ over $I$, then $\psi$ tamely degenerates to $L$ (see \cite[Remark 3.3]{AF}). Therefore, $\psi$ needs to tamely degenerate to a lightlike line segment if $\Sigma$ extends across the lightlike line segment.

\subsection{Lightlike line with shrinking singularities.}\label{sec:2.2}
Let $X\colon \mathbb{D}\to \Sigma={\rm graph}(\psi)$ be an isothermal parametrization from the unit disk $\mathbb{D}=\{w\in \mathbb{C}\mid |w|<1\}$. The spacelike condition of $\Sigma$ implies that $\Sigma$ is bounded since $\psi$ is a locally $1$-Lipschitz function over a bounded domain. Thus, $X$ is a bounded harmonic mapping, and therefore can be written as a Poisson integral of some bounded mapping $\widehat{X}\colon \partial \mathbb{D} \to \mathbb{L}^3$ (see \cite[p.72, Lemma 1.2]{Katz}), that is,
\[
	X(w)=P^{\mathbb{D}}_{\widehat{X}}(w):=\frac{1}{2\pi}\int^{2\pi}_0 \frac{1-|w|^2}{|e^{it}-w|^2}\ \widehat{X}(e^{it})dt.
\]
It is well-known that if $w_0 \in \partial \mathbb{D}$ is a jump point of $\widehat{X}$, that is, the one-sided limits
\[
	a:={\rm ess.}\! \lim_{\hspace*{-4ex} 0<t\to 0} \widehat{X}(w_0 e^{-it}),\ \ b:={\rm ess.}\! \lim_{\hspace*{-4ex} 0<t\to 0} \widehat{X}(w_0 e^{it})
\]
exist and are different, then the cluster point set $C(X, w_0)$ of $X$ at $w_0$ becomes a straight line segment joining $a$ and $b$. In \cite[Theorem 1.1]{AF}, the authors showed that if $\psi$ tamely degenerates to a lightlike line segment $L$, then there exists a jump point $w_0\in \partial \mathbb{D}$ of $\widehat{X}$ such that $L\subset C(X,w_0)$. In other words, the boundary lightlike line segment corresponds to a single point under an isothermal parametrization.

\begin{definition}\label{def:LLLwithShrinkSing}
Let $L\subset \partial \Sigma$ be a lightlike line segment and $a,b\in\mathbb{L}^3$ ($a\neq b$) its endpoints. We say that $L$ has \textit{shrinking singularities at the endpoints} if there exists $w_0=e^{it_0} \in\partial \mathbb{D}$ such that $\widehat{X}(e^{it})\equiv a$ for all $t\in (t_0-\varepsilon,t_0)$ and $\widehat{X}(e^{it})\equiv b$ for all $t\in (t_0,t_0+\delta)$ for some $\varepsilon,\delta >0$.
\end{definition}

\begin{remark}
The map $\widehat{X}$ is determined up to sets of measure zero. Thus, Definition \ref{def:LLLwithShrinkSing} requires that $\widehat{X}$ satisfies the conditions after an appropriate modification on a set of measure zero.
\end{remark}

\noindent
In this case, the isothermal parametrization $X$ extends to a harmonic mapping across each of the two arcs $(t_0-\varepsilon,t_0)$ and $(t_0,t_0+\delta)$ by the Schwarz reflection principle. The extended map becomes a {\it generalized maximal surface in the sense of \cite{ER}} and has shrinking singularities at $a$ and $b$. Here, the definitions of the generalized maximal surfaces and the shrinking singularities were introduced in \cite{ER} and \cite{KY}, respectively. Furthermore, by \cite[Corollary 3.7]{AF}, we can see $\psi$ tamely degenerates to $L$. 

There are useful ways to check whether the endpoints of a lightlike line segment are shrinking singularities or not from the shape of $\Sigma$. Such methods will be given in Section \ref{subsec:findShrinkSing}.

\subsection{Reflection principle for lightlike line segments.}\label{sec:2.3}
Recall that a lightlike line segment $L\subset \partial \Sigma$ corresponds to a single point under an isothermal parametrization. Thus, this isothermal parametrization seems to be quite useless to construct an extension of $\Sigma$ across $L$. However, we introduce a modification method of the isothermal parameter like a ``blow-up'', and consequently we will prove that $\Sigma$ can be extended across $L$ with some reflection symmetry if $L$ has shrinking singularities at its endpoints.

To explain the modification, it is useful to consider the Poisson integral for the upper half plane $\mathbb{H}=\{\zeta\in \mathbb{C}\mid {\rm Im}(\zeta)>0\}$. For a bounded piecewise continuous function $U\colon \mathbb{R}\to \mathbb{R}$, the Poisson integral is defined by
\[
	P_U(\zeta)=\frac{1}{\pi}\int_{-\infty}^{\infty}\frac{\eta}{(\xi-s)^2+\eta^2}U(s)ds,
\]
where $\zeta=\xi+i\eta$. We remark that the Poisson integrals for the upper half plane and the unit disk are related by the M\"obius transformation $\Phi(\zeta)=(\zeta-i)/(\zeta+i)$, more precisely, $P_U(\Phi^{-1}(w))=P^{\mathbb{D}}_{U\circ \Phi^{-1}}(w)$ holds (see \cite[Chap.6]{Ahlfors} for basic properties of the Poisson integrals). Let $I=(\sigma,\tau)\ (\sigma<\tau)$ be an interval and $\chi_I$ its characteristic function. Then we can easily see that the harmonic measure of $I$ is given by
\begin{align*}
	P_{\chi_I}(\zeta)&=\frac{1}{\pi}\left\{ {\rm Arg}(\tau-\zeta)-{\rm Arg}(\sigma-\zeta)\right\}\\
									&=\frac{1}{\pi}\left\{ {\rm Arg}^+(\tau-\zeta)-{\rm Arg}^+(\sigma-\zeta)\right\}.
\end{align*}
Here, $\Arg$ and $\Arg^+$ denote the harmonic blanches of $\arg$ defined on $\mathbb{C}\setminus (-\infty,0]$ and $\mathbb{C}\setminus [0,\infty)$ which take values in $(-\pi,\pi)$ and $(0,2\pi)$, respectively. Note that
\begin{equation}
	\Arg^+(-\zeta)=\Arg(\zeta)+\pi \label{eq:arg_arg+}
\end{equation}
holds for $\zeta\in\mathbb{C}\setminus(-\infty,0]$.

\begin{theorem}[Reflection principle for lightlike lines] \label{thm:Main}
If a lightlike line segment $L\subset \partial \Sigma$ has shrinking singularities at its endpoints $a, b\in \mathbb{L}^3$, then $\Sigma$ extends to a surface with zero mean curvature across $L$. 

Further, the extended surface is invariant under the point symmetry with respect to the midpoint $c=(a+b)/2$ of $L$. In particular, it is maximal, except along $L$.
\end{theorem}

\begin{proof}
Let $X\colon \mathbb{H}\to \Sigma$ be an isothermal parametrization from the upper half plane. Then, $X$ can be written as a Poisson integral of some bounded map $\widehat{X}\colon \mathbb{R}\simeq \partial{\mathbb{H}}\to \mathbb{L}^3$. In this case, the assumption that $L$ has shrinking singularities at its endpoints $a$ and $b$ is translated into the fact that there exists $s_0\in \mathbb{R}$ such that $\widehat{X}\equiv a$ on an interval $(s_0 +\sigma,s_0)$ and $\widehat{X}\equiv b$ on an interval $(s_0,s_0+\tau)$, where $\sigma<0<\tau$. By composing a M\"obius transformation, we may assume $s_0=0$.

Let $\Pi\colon D^+:=\mathbb{R}_{>0}\times (0,\pi)\to \mathbb{H}$ be a homeomorphism defined by $\Pi(r,\theta)=re^{i\theta}$. Observe that $\Pi$ is real analytic on a domain $D:=\mathbb{R}\times (0,\pi)$ wider than $D^+$, and $\Pi(D)=\mathbb{H}\cup \{0\} \cup \mathbb{H}^{\ast}$, where $\mathbb{H}^{\ast}$ denotes the lower half plane. We prove that the real analytic map $X\circ \Pi$ on $D^+$ extends to $D$ real analytically. First, the linearity of the Poisson integral implies
\begin{align*}
	X(\zeta)=&P_{\widehat{X}}(\zeta)=a P_{\chi_{(\sigma,0)}}(\zeta)+bP_{\chi_{(0,\tau)}}(\zeta)+P_W(\zeta)\\
						&=\frac{a}{\pi}\left\{\Arg^+(-\zeta)-\Arg^+(\sigma-\zeta)\right\}+\frac{b}{\pi}\left\{\Arg(\tau-\zeta)-\Arg(-\zeta)\right\}+P_W(\zeta),
\end{align*}
where $\zeta=\Pi(r,\theta)$ and $W=(1-\chi_{(\sigma,\tau)})\widehat{X}$. By (\ref{eq:arg_arg+}), we have $\Arg^+(-\zeta)=\Arg(re^{i\theta})+\pi=\theta+\pi,$ and $\Arg(-\zeta)=\Arg^+(re^{i\theta})-\pi=\theta-\pi$. Thus,
\[
	X(\zeta)=a+b+\frac{a-b}{\pi}\theta +\frac{b}{\pi}\Arg(\tau-\zeta)-\frac{a}{\pi}\Arg^+(\sigma-\zeta)+P_W(\zeta).
\]
First three terms are clearly real analytic on $D$. Further since $\left(\sigma-\Pi(D)\right)\subset \mathbb{C}\setminus [0,+\infty)$ and $\left(\tau-\Pi(D)\right)\subset \mathbb{C}\setminus(-\infty,0]$, the fourth and the fifth terms are also real analytic on $D$. Finally, observe that $P_W$ is harmonic on $\mathbb{H}$, continuous on $\mathbb{H}\cup (\sigma,\tau)$ and $P_W\equiv (0,0,0)$ on $(\sigma,\tau)$. Thus the Schwarz reflection principle implies that $P_W$ extends to $\mathbb{H}\cup (\sigma, \tau) \cup \mathbb{H}^{\ast}$ harmonically by $P_W(\overline{\zeta})=-P_W(\zeta)$. Since $\Pi(D)\subset \mathbb{H}\cup (\sigma, \tau) \cup \mathbb{H}^{\ast}$, we conclude that $P_W\circ\Pi$ and therefore $X\circ\Pi$ are real analytic on $D$. Thus, $\Sigma$ extends to a surface with zero mean curvature across $L$.

For the latter statement, note that $\Pi(-r,\pi-\theta)=re^{-i\theta}=\overline{\zeta}$ for $\zeta=\Pi(r,\theta)$. Thus
\begin{align*}
	X\circ\Pi(-r,\pi-\theta)&=a+b+\frac{a-b}{\pi}(\pi-\theta)\\
						&\ \ \ \ \ \ \ \ \  +\frac{b}{\pi}\Arg\left(\overline{(\tau-\zeta)}\right)-\frac{a}{\pi}\Arg^+\left(\overline{(\sigma-\zeta)}\right)+P_W(\overline{\zeta})\\
						&=a+b+a-b-\frac{a-b}{\pi}\theta\\
						&\ \ \ \ \ \ \ \ \  +\frac{b}{\pi}\left(-\Arg(\tau-\zeta)\right)-\frac{a}{\pi}\left(2\pi-\Arg^+(\sigma-\zeta)\right)-P_W(\zeta)\\
						&=-\frac{a-b}{\pi}\theta-\frac{b}{\pi}\Arg(\tau-\zeta)+\frac{a}{\pi}\Arg^+(\sigma-\zeta)-P_W(\zeta).		
\end{align*}
We have $X\circ \Pi(-r,\pi-\theta)+X\circ\Pi(r,\theta)=a+b$. This implies the desired reflection symmetry.
\end{proof}

\begin{figure}[htb]
 \begin{center}
 \includegraphics[height=3.1cm]{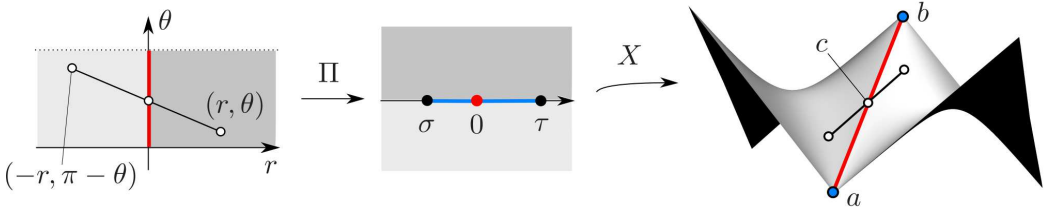} 
\caption{Blow-up of a discontinuous point of $\widehat{X}$ and the real analytic extension $X\circ \Pi$ across a boundary lightlike line segment.}\label{fig:extension} 
%\label{fig:Periodic}
 \end{center}
\end{figure}

\begin{remark} In the above situation, it holds that
\[
	X\circ\Pi(0,\theta)=b+\frac{a-b}{\pi}\theta=a\frac{\theta}{\pi}+b\left(1-\frac{\theta}{\pi}\right).
\]
Therefore, $X\circ \Pi(0,\theta)$ is the dividing point of $L$ which divides $L$ into two segments with Euclidean lengths $(1-\theta/\pi):\theta/\pi$. 
\end{remark}

\subsection{Methods of finding shrinking singularities on lightlike lines.} \label{subsec:findShrinkSing}
At the end of this section, we give some useful criteria for the endpoint of a lightlike line segment boundary to be a shrinking singularity. By using this, we can find shrinking singularities on lightlike line segments from the shape of the surfaces. 

Let $L\subset \partial \Sigma$ be a lightlike line segment to which $\psi$ tamely degenerates. Further, let $X\colon \mathbb{D}\to \Sigma={\rm graph}(\psi)$ be an isothermal parametrization, and $\widehat{X}\colon \partial \mathbb{D}\to \mathbb{L}^3$ its boundary value function, that is, $X(w)=P^{\mathbb{D}}_{\widehat{X}}(w)$. Then, as mentioned above, there is a jump point $w_0 \in \partial \mathbb{D}$ of $\widehat{X}$ such that $L\subset C(X,w_0)$. If $\widehat{X}$ maps some arc of one side of $w_0$ constantly to an endpoint of $L$, we say that the endpoint is a \textit{shrinking singularity}. Then, $L$ has shrinking singularities at its endpoints if and only if both of the two endpoints are shrinking singularities.

One of the most typical situations where shrinking singularities appear on lightlike line segments is as follows.

\begin{fact}[{\cite[Theorem 4.5]{AF}}] \label{fact:2LLL}
Assume that $\partial \Sigma$ contains two adjacent lightlike line segments $L_1$ and $L_2$ whose union does not form a straight line segment. If $\psi$ tamely degenerates to $L_1$ and $L_2$, respectively, then the common endpoint of $L_1$ and $L_2$ is a shrinking singularity. 

\end{fact}

\noindent
More precisely, Fact \ref{fact:2LLL} says that if we let $w_1, w_2\in \partial \mathbb{D}$ be corresponding jump points to $L_1$ and $L_2$, respectively, then $w_1\neq w_2$ and one of the two arcs of $\partial \mathbb{D}$ joining $w_1$ and $w_2$ is constantly mapped to the common endpoint of $L_1$ and $L_2$ by $\widehat{X}$, see Figure \ref{fig:2lines}.

\begin{figure}[htb]
 \begin{center}
  \includegraphics[height=3.6cm]{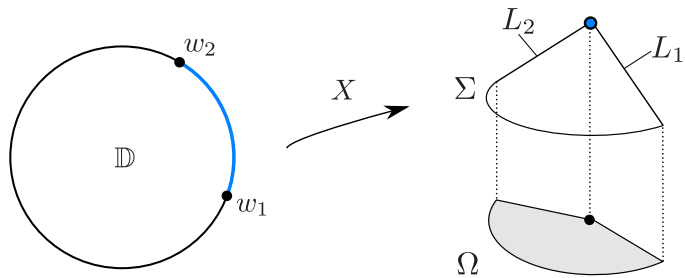} 
\caption{A situation to which Fact \ref{fact:2LLL} is applicable.}\label{fig:2lines} 
 \end{center}
\end{figure}

We shall remark that the same is true even if $\Omega$ has a slit (see Figure \ref{fig:withSlit}).

\begin{proposition}\label{prop:slitVer}
Let $I\subset \Omega$ be an open Euclidean segment joining an interior point $p\in \Omega$ and a boundary point. Assume that a maximal graph $\Sigma={\rm graph}(\psi)$ over a domain $\Omega \setminus \overline{I}$ contains two lightlike line segments $L_1$ and $L_2$  (possibly coinciding with each other) in its boundary over the slit $I$. If $\psi$ tamely degenerates to $L_1$ and $L_2$, respectively, then the common endpoint of $L_1$ and $L_2$ over $p$ is a shrinking singularity.
\end{proposition}

\begin{figure}[htb]
 \begin{center}
  \includegraphics[height=3.6cm]{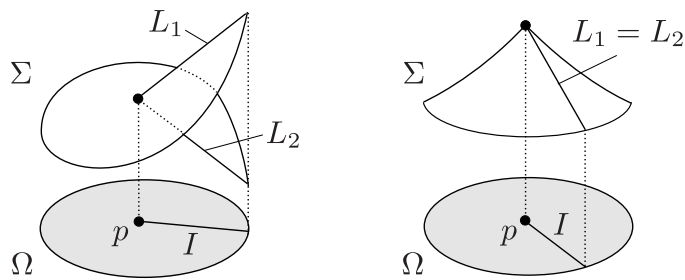} 
\caption{Situations to which Proposition \ref{prop:slitVer} are applicable.}\label{fig:withSlit} 
 \end{center}
\end{figure}

\begin{proof}
The proof is completely the same as \cite[Theorem 4.5]{AF}. To do this, we only need to check that the Hengartner-Schober theorem \cite[Theorem 4.3]{HS}, \cite[Lemma 3.4]{AF} can be applied to $\Omega\setminus \overline{I}$. However, in fact, this theorem can be applied to any bounded simply connected domain with locally connected boundary as stated in their original paper \cite[Theorem 4.3]{HS} (see also \cite[Section 3.3]{D}).
\end{proof}

We explain how to use Fact \ref{fact:2LLL} and Proposition \ref{prop:slitVer} with examples below: 

\begin{example}\label{example:knownperiodic}
The first example is a singly periodic {\it maximal surface of Riemann type} illustrated on the left side of Figure \ref{fig:Periodic}, which is given by
\begin{equation*}
\mathcal{S}_1=\{ (x,y,t)\in \mathbb{L}^3 \mid 2(-x + t)\sin{t}-(x^2 + y^2 - 2 x t+ t^2)\cos{t}=0\},
\end{equation*}
 (see \cite[Theorem 5.3 (1-i)]{A}). Choose one sheet of $\mathcal{S}_1$ as on the left of Figure \ref{fig:example1}. Then the sheet is an entire graph which is maximal except at two points that are isolated singularities and the lightlike line segment joining them. If we restrict the graph to a domain $\Omega_1$ like in the right side of Figure \ref{fig:example1}, then the restricted maximal graph tamely degenerates to the lightlike line segment from each side, since it extends to the entire graph. Thus Proposition \ref{prop:slitVer} implies that the lightlike line segment has shrinking singularities at its endpoints. Consequently, $\mathcal{S}_1$ is invariant under the point symmetry at the midpoint of the lightlike line segment by Theorem \ref{thm:Main}. 
Note that maximal surfaces of Riemann type without lightlike lines were determined via the Weierstrass representation by L\'opez-L\'opez-Souam \cite{LLS}. 
 
\begin{figure}[htb]
 \begin{center}
  \includegraphics[height=3.8cm]{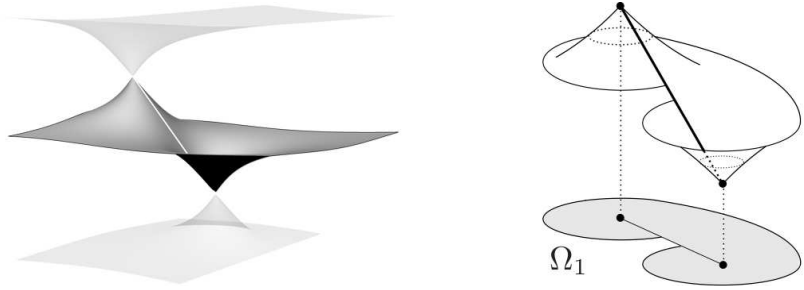} 
\caption{One sheet of $\mathcal{S}_1$ (left) and a restricted graph (right).}\label{fig:example1} 
 \end{center}
\end{figure}

The second and third examples are the following doubly and triply periodic maximal surfaces illustrated in the center and on the right of Figure \ref{fig:Periodic}, respectively:
\begin{align*}
\mathcal{S}_2&=\{ (x,y,t)\in \mathbb{L}^3 \mid \cos{t} \cosh{y}+ \cos{x}=0\},\\
 \mathcal{S}_3&=\{ (x,y,t)\in \mathbb{L}^3 \mid \cos{t} - \cos{x} \cos{y}=0\}.
\end{align*}
For instance, we similarly take one sheet of $\mathcal{S}_2$ as on the left of Figure \ref{fig:example2}. Then the sheet is an entire maximal graph with isolated singularities and lightlike line segments. In this case, for each isolated singularity there are two lightlike line segments joining it, and thus Fact \ref{fact:2LLL} is applicable to the graph restricted to a domain $\Omega_2$ as on the right side of Figure \ref{fig:example2}, by the same argument as for $\mathcal{S}_1$. Therefore, each lightlike line segment has shrinking singularities at its endpoints, and $\mathcal{S}_2$ is invariant under the point symmetry about this midpoint. 

\begin{figure}[htb]
 \begin{center}
  \includegraphics[height=3.8cm]{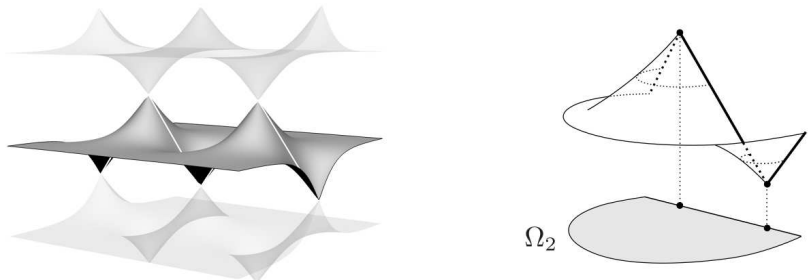} 
\caption{One sheet of $\mathcal{S}_2$ (left) and a restricted graph (right).}\label{fig:example2} 
 \end{center}
\end{figure} 

The same is true for the surface $\mathcal{S}_3$. Note that $\mathcal{S}_3$ is called the {\it spacelike Scherk surface} in \cite{CR}.  For each shrinking singularity on this surface there are four lightlike lines passing through it (see Figure \ref{fig:Periodic}, right). Further, this surface also can be constructed by the method described in Section \ref{sec:3} (apply the tessellation generated by a square). We also remark that periodic maximal surfaces with shrinking singularities but without lightlike line segments were intensively studied in \cite{FLS,FL2,LLS} (see also their references). 
\end{example}

%===============================================================SECTION Triply Periodic===========
\section{Construction of periodic maximal surfaces with lightlike lines}\label{sec:3}
The reflection principle in Theorem \ref{thm:Main} gives a new technique to extend and construct maximal surfaces with lightlike line segments. 
As an application of Theorem \ref{thm:Main}, we construct triply periodic maximal surfaces with lightlike lines.

\subsection{Procedures of construction}\label{subsec:3.1}
The procedures are as follows:

\noindent {\bf Step 1}. Give a tessellation of $\mathbb{R}^2$ which is made by a polygon $\Omega$ satisfying the following two conditions (see Figure \ref{fig:SketchofSteps}, top left):
\begin{itemize}
\item[(1-a)] there exists a maximal graph $\Sigma_1$ which tamely degenerates to future and past-directed lightlike line segments on the edges of $\Omega$ alternately. 
\item[(1-b)] the tessellation is generated by point-symmetries of $\Omega$ with respect to its midpoints of edges. 
\end{itemize}

\noindent{\bf Step 2}. Construct a doubly periodic entire maximal graph $\Sigma_2$ from $\Sigma_1$ by iterating the reflection principle for lightlike line segments. The period lattice $\Lambda$ corresponds to that of the tessellation in Step (1-b) (see Figure \ref{fig:SketchofSteps}, top right). 

\noindent{\bf Step 3}. Construct a triply periodic proper maximal surface $\Sigma_3$ from $\Sigma_2$ by iterating the reflection principle for shrinking singularities. The period lattice corresponds to basis of $\Lambda$ and a lightlike vector along a lightlike line segment (see Figure \ref{fig:SketchofSteps}, bottom). 

\begin{figure}[h!]
 \begin{center}
  \includegraphics[scale=0.3]{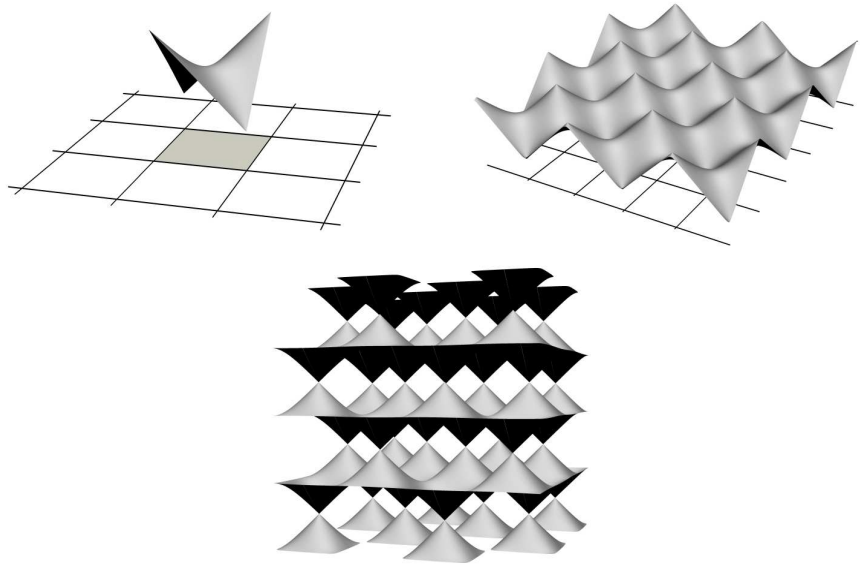} 
\caption{The surfaces $\Sigma_1$ in Step 1 (top left), $\Sigma_2$ in Step 2 (top right) and $\Sigma_3$ in Step 3 (bottom).}\label{fig:SketchofSteps} 
 \end{center}
\end{figure}

\subsection{Details and examples}\label{subsec:3.2}\ \newline
\noindent
{\bf Step (1-a)}:
One way to construct maximal graphs with lightlike line boundaries in Step (1-a) is to use Jenkins and Serrin's criteria \cite{JS} for infinite boundary value problems of the minimal surface equation in the Euclidean 3-space: 

Let $\Omega\subset \mathbb{R}^2$ be a polygonal domain whose boundary consists of a finite number of open line segments $A_1,\ldots,A_k,B_1,\ldots,B_l$. For each of families $\{A_j\}$ and $\{B_j\}$, assume that no two of the elements meet to form a convex corner. Further, for a polygonal domain $P\subset \Omega$ whose vertices are those of $\Omega$, let $\alpha_P$ and $\beta_P$ denote respectively, the total length of $A_j$ such that $A_j\subset\partial P$ and the total length of $B_j$ such that $B_j\subset\partial P$, and let $\gamma_P$ be the perimeter of $P$. 

Then, by the duality of boundary value problems for minimal surfaces in $\mathbb{E}^3$ and maximal surfaces in $\mathbb{L}^3$ proved in \cite[Theorem 1]{AF}, the classical Jenkins-Serrin's theorem \cite[Theorem 3]{JS} yields the following result.

\begin{fact}\label{fact:JS}
There exists a maximal graph $\Sigma_1$ over $\Omega$ which tamely degenerates to a future-directed lightlike line segment on each $A_j$ and a past-directed lightlike line segment on each $B_j$ if and only if 
\begin{equation}\label{eq:JS_condition1}
		2\alpha_P<\gamma_P \ \ \ \text{and}\ \ \  2\beta_P <\gamma_P
\end{equation}
hold for each polygonal domain $P\subsetneq \Omega$ taken as above and 
\begin{equation}\label{eq:JS_condition2}
		\alpha_{\Omega}=\beta_{\Omega}
\end{equation}
holds. The solution is unique up to an additive constant if it exists.
\end{fact}

\begin{remark}
Each maximal graph over a polygonal domain $\Omega$ in Fact \ref{fact:JS} can be constructed by using the Poisson integral of a step function on $S^1\simeq \partial{\mathbb{D}}$ valued in the vertices of $\Omega$, which is a good way to construct maximal graphs explicitly. See \cite[Corollary 3.12 and Section 4.4]{AF} for more details.
\end{remark}

\noindent {\bf Step (1-b)}: Let $\mathcal{M}$ be the set of polygonal domains which tessellate $\mathbb{R}^2$ by taking point symmetries with respect to the midpoints of all edges repeatedly. 

The following assertions give a classification of the shapes of domains in $\mathcal{M}$ satisfying the conditions in Fact \ref{fact:JS}. See Figure \ref{fig:tiling}. 

\begin{lemma}\label{lemma:zonogons}
 Each $\Omega \in \mathcal{M}$ is either one of the following:
\begin{quote}
\setlength{\leftskip}{-1.1cm}
{\rm (i)} a triangle, {\rm (ii)} a quadrilateral, or \\
{\rm (iii)} a hexagon whose opposite sides are parallel and of equal length.
\end{quote}
\end{lemma}

\begin{lemma}\label{lemma:zonogons2}
A domain $\Omega \in \mathcal{M}$ satisfying \eqref{eq:JS_condition1} and \eqref{eq:JS_condition2} is either 
\begin{quote}
\setlength{\leftskip}{-1.1cm}
{\rm (ii$^\prime$)} a quadrilateral whose two pairs of opposite sides are equal in length, or \\
{\rm (iii$^\prime$)} a hexagon satisfying the conditions {\rm(iii)} in Lemma $\ref{lemma:zonogons}$ and 
\begin{equation}\label{eq:JS_condition_6-gon}
d>|b-(a+c)|,
\end{equation}
where $a,b,c$ are the lengths of consecutive three edges and $d$ is the smallest length of diagonal lines connecting opposite vertices of the hexagon. 
\end{quote}
\end{lemma}

\begin{figure}[htb]
 \begin{center}
  \includegraphics[height=5cm]{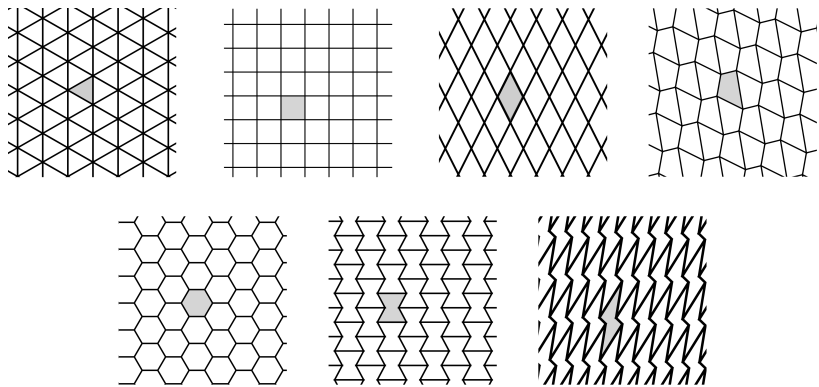} 
\caption{Tessellations of $\mathbb{R}^2$ from polygonal domains in $\mathcal{M}$. }
\label{fig:tiling} 
 \end{center}
\end{figure}

\noindent The proofs are given in Appendix A. Therefore, we can construct the maximal graph $\Sigma_1$ over a given polygonal domain $\Omega$ in Lemma \ref{lemma:zonogons2} (see Figure \ref{Hexagon1}). 

\begin{figure}[htb]
 \begin{center}
  \includegraphics[height=3.cm]{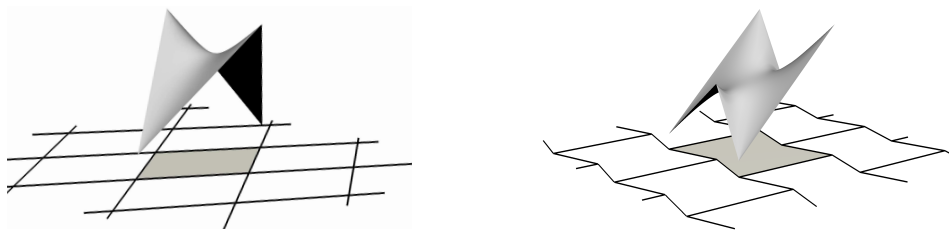} 
\caption{Maximal graphs with lightlike line boundaries over polygonal domains in $\mathcal{M}$ generated by Step 1.} \label{Hexagon1} 
 \end{center}
\end{figure}
\vspace{0.2cm}

\noindent {\bf Step 2}: By Fact \ref{fact:2LLL}, the endpoints of each lightlike line segment are shrinking singularities. Applying Theorem \ref{thm:Main}, i.e. taking point symmetries at the midpoints of the lightlike line segments, we have a doubly periodic maximal graph $\Sigma_2$ with shrinking singularities on the vertices and lightlike line segments over the edges of the polygons (see Figure \ref{fig:Hexagon2}).

\begin{figure}[htb]
 \begin{center}
  \includegraphics[height=3.3cm]{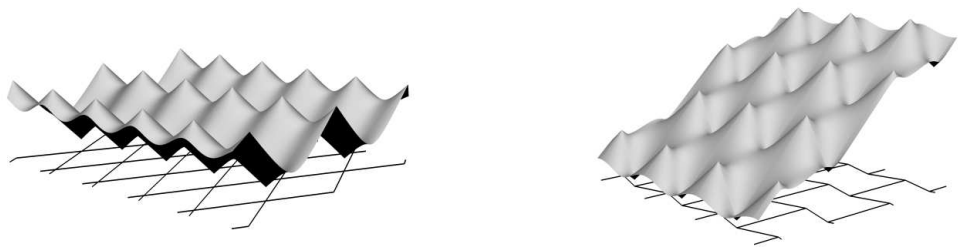} 
\caption{Doubly periodic entire maximal graphs with lightlike line segments and shrinking singularities generated by Step 2.} \label{fig:Hexagon2}
 \end{center}
\end{figure}
\vspace{0.2cm}

\noindent {\bf Step 3}: Applying the reflection principle for shrinking singularities, i.e.~ taking point symmetries there, we have a triply periodic proper maximal surface $\Sigma_3$, which is a multi-valued graph with infinitely many sheets congruent to $\Sigma_2$ 
(see Figure \ref{fig:Hexagon3}).

\begin{figure}[htb]
 \begin{center}
  \includegraphics[height=4.0cm]{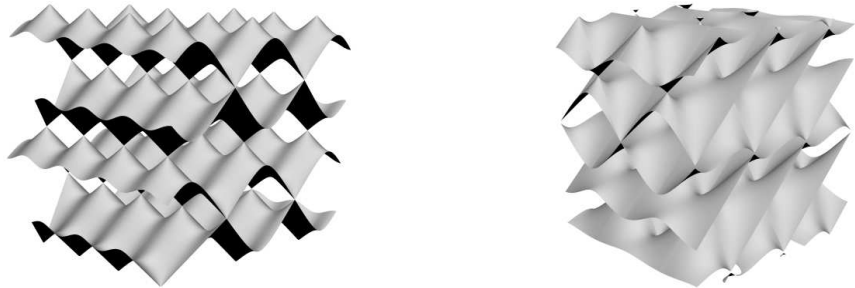} 
\caption{Triply periodic maximal surfaces with lightlike line segments and shrinking singularities generated by Step 3.} \label{fig:Hexagon3}
 \end{center}
\end{figure}

\subsection{Parametric representation of maximal surfaces with lightlike lines}
Well-known classes of maximal surfaces with singularities such as generalized maximal surfaces in \cite{ER} and {\it maxfaces} in \cite{UY1} are defined on Riemann surfaces. 
Unfortunately, isothermal coordinates break down near lightlike line segments for the periodic maximal surfaces $\Sigma_2$ and $\Sigma_3$ in the previous subsection. Therefore, to parametrize them we need to consider a wider class of maximal surfaces.

\begin{definition}[{\cite[Definition 2.1]{UY1}}]
A smooth map $X\colon M \to \mathbb{L}^3$ from a 2-dimensional manifold $M$ to $\mathbb{L}^3$ is called a {\it maximal map} if there is an open dense set $W\subset M$ such that $X|_W$ is a spacelike maximal immersion. 
\end{definition}

Suppose the boundary of a maximal graph contains a lightlike line segment with shrinking singularities at its endpoints. Then, as shown in the proof of Theorem \ref{thm:Main}, the extended surface across the lightlike line segment can be parametrized by a non-conformal but real analytic mapping (see Figure \ref{fig:extension}).  
Moreover, the extension with respect to each of the shrinking singularities is parametrized by a generalized maximal surface, in particular, an analytic mapping. As a consequence, for example, the triply periodic maximal surfaces constructed in Section \ref{subsec:3.1} can be parametrized by a (non-conformal) real analytic proper maximal map by 
gluing the above parametrizations together. In this case, the domain of the maximal map is a 2-dimensional manifold of infinite genus.
See Figure \ref{fig:Maximal_map}. 

\begin{figure}[htb]
 \begin{center}
  \includegraphics[height=9.0cm]{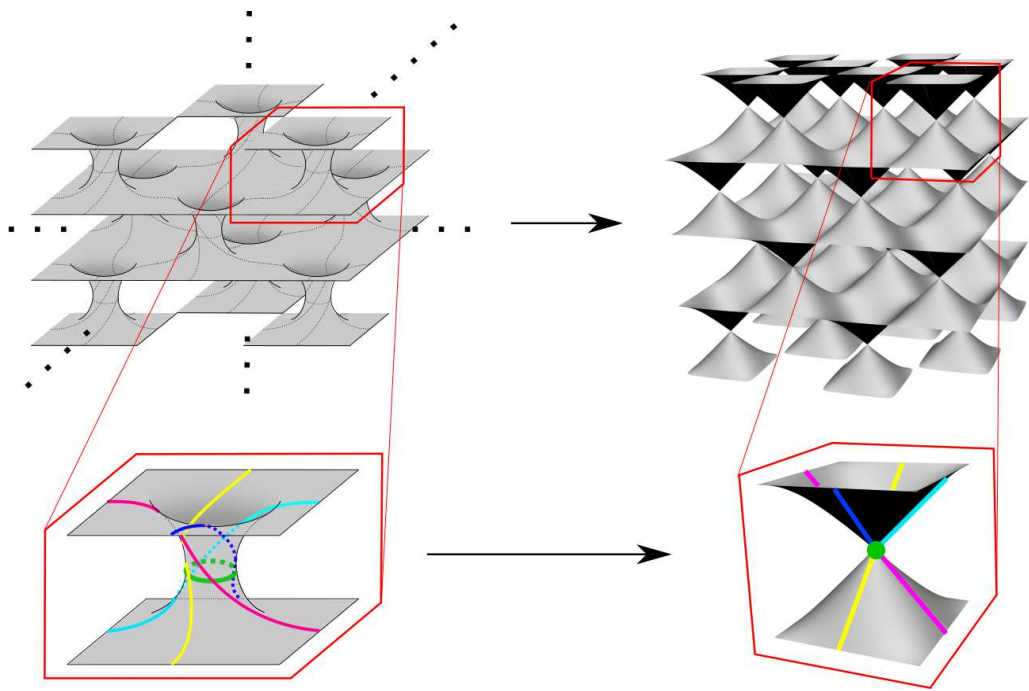} 
\caption{The surface $\Sigma_3$ parametrized by a proper maximal map.}\label{fig:Maximal_map}
 \end{center}
\end{figure}

\section{Concluding remarks and a future problem}\label{sec:4}
Finally, we remark that the reflection principle in Theorem \ref{thm:Main} is valid only for lightlike line segments connecting shrinking singularities. %In addition to this situation, 
Therefore, it is natural to consider the following cases, as well: The boundary of a maximal graph contains

\begin{quote}
{\bf Case 1} an entire lightlike line, or \quad {\bf Case 2} a lightlike half-line.
\end{quote}

On the contrary, Theorem \ref{thm:Main} is not valid for such entire lightlike lines or lightlike half-lines, in particular, we cannot take the ``midpoint'' of any such lines. See Figure \ref{fig:cats}, for example.
Typical examples of Case 2 are the following hyperbolic catenoid $\mathcal{H}$ and the parabolic catenoid $\mathcal{P}$, which are invariant under rotations in $\mathbb{L}^3$ with respect to spacelike and lightlike axes, respectively (see Figure \ref{fig:cats} and also \cite{CR,Okayama} for their implicit representations):
\begin{align*}
\mathcal{H}&=\{ (x,y,t)\in \mathbb{L}^3 \mid \sin^2{x}+y^2-t^2=0\},&\\
\mathcal{P}&=\{ (x,y,t)\in \mathbb{L}^3 \mid 12(x^2- t^2)-(x + t)^4 + 12y^2=0\}.&
\end{align*}
\begin{figure}[htb]
 \begin{center}
        \includegraphics[height=4.0cm]{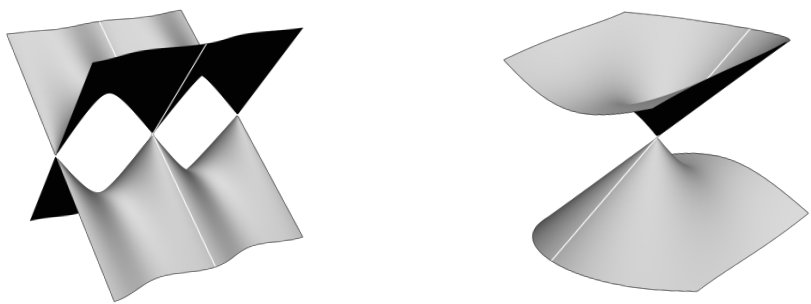}
\caption{The hyperbolic catenoid (left) and the parabolic catenoid (right).} \label{fig:cats}
 \end{center}
\end{figure}
Also, examples of Case 1 were given in \cite{AUY2}, and it is known that such a maximal graph containing an entire lightlike line  cannot be defined on a convex domain, as proved in \cite[Lemma 2.1]{FL}.

 From the viewpoint of the function theory, in the present article we discussed bounded harmonic mappings. On the other hand, Cases 1 and 2 lead us to considering {\it unbounded} harmonic mappings for which the Poisson integrals do not work. 
 
Related to this, the following problem remains as a future work.

\begin{question} Is there a reflection principle for entire lightlike lines or lightlike half-lines (as in the above Cases $1$ and $2$) on the boundaries of maximal surfaces?
\end{question}

%===============================================================APPENDIX=====================
\appendix 
\section{Proofs of Lemmas \ref{lemma:zonogons} and \ref{lemma:zonogons2}}
In this appendix, we give a classification of tessellations of $\mathbb{R}^2$ satisfying the conditions of Step 1 in Section \ref{subsec:3.1}. 

For a chosen vertex $z$ of a polygonal domain $\Omega$, we denote the interior angles of $\Omega$ starting from $z$ by $\alpha_1,\alpha_2,\ldots,\alpha_n$ and set $\alpha_{mn+k}=\alpha_k$ for a positive integer $m$ and $k=1,2,\ldots,n$. 

\begin{lemma}\label{lemma:A1}
If $\Omega \in \mathcal{M}$, then there exists a unique $j\geq 2$ such that $\alpha_1+\alpha_2+\cdots+\alpha_j=2\pi$ and $\alpha_l=\alpha_m$ if $l\equiv k$ mod $j$. Moreover, this $j$ is independent of a choice of vertices of $\Omega$.
\end{lemma}
We call the above $j$ for $\Omega\in \mathcal{M}$ the {\it valency} of $\Omega$, and in this case each vertex is said to be {\it j-valent}.

\begin{proof}
Let us consider the edge connecting the vertices with the angles $\alpha_1$ and $\alpha_2$. If we take the reflection with respect to the midpoint of the edge, then $\alpha_2$ appears next to $\alpha_1$ around $z$. Inductively, $\alpha_1, \alpha_2, \alpha_3, \ldots$ appear around $z$ in this order (see Figure \ref{fig:branch_condition}). Thus there is some $j$ such that $\alpha_1+\alpha_2+\cdots+\alpha_j=2\pi$ and $\alpha_l=\alpha_m$ if $l\equiv k$ mod $j$.

Next, we denote the valencies of the vertices with angles $\alpha_1$ and $\alpha_2$ by $j_1$ and $j_2$.  
Then $\alpha_1+\alpha_2+\cdots+\alpha_{j_1}=2\pi$ and $\alpha_{j_1+1}=\alpha_1$ hold. Hence we have 
\begin{equation*}
\alpha_2+\alpha_3+\cdots+\alpha_{j_1+1}=\alpha_1+\alpha_2+\cdots+\alpha_{j_1}=2\pi,
\end{equation*} 
which implies $j_1=j_2$. By induction on the vertices, we obtain the last assertion.
\end{proof}

\begin{figure}[htb]
 \begin{center}
  \includegraphics[height=5.0cm]{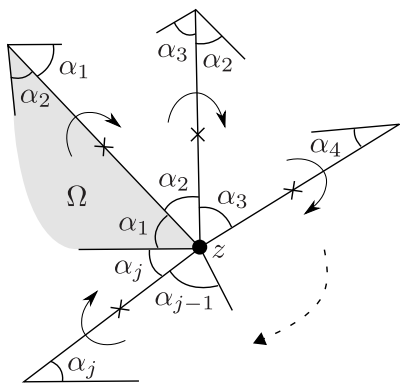} 
\caption{The condition $\alpha_1+\alpha_2+\cdots+\alpha_j=2\pi$ at the vertex $z$.} 
\label{fig:branch_condition}
 \end{center}
\end{figure}

\begin{lemma}\label{lemma:A2}
Let $\Omega \in \mathcal{M}$ be an $n$-gon and each vertex is $j$-valent for $j\geq 2$. Then the following equation holds.
\begin{equation*}
(n-2)j=2n.
\end{equation*}
In particular, the possible values of $n$ and $j$ are 
\begin{equation*}
(n,j)=(3,6),(4,4),(6,3).
\end{equation*}

\end{lemma}
\begin{proof}
By Lemma \ref{lemma:A1}, the relation $\alpha_i+\alpha_{i+1}+\cdots+\alpha_{i+j-1}=2\pi$ and $\alpha_{i+j}=\alpha_{i}$ hold for arbitrary $i$. Since the sum of the interior angles of an $n$-gon is $(n-2)\pi$, we obtain
\begin{align*}
j(n-2)\pi&=j(\alpha_1+\alpha_{2}+\cdots+\alpha_{n})\\
&=\sum_{i=1}^n(\alpha_i+\alpha_{i+1}+\cdots+\alpha_{i+j-1})\\
&=2n\pi,
\end{align*}
which is the desired equality.
\end{proof}

It can be easily seen that any triangles and quadrilaterals are in $\mathcal{M}$, and they are $6$-valent and $4$-valent, respectively.
When $\Omega \in \mathcal{M}$ is a hexagon, which is $3$-valent, Lemma \ref{lemma:A1} yields that the interior angles of $\Omega$ are written as $\alpha_1,\alpha_2,\alpha_3,\alpha_1,\alpha_2,\alpha_3$ in this order. This implies that the opposite sides of edges of $\Omega$ are parallel and of equal length, and hence we have a proof of Lemma \ref{lemma:zonogons}. 

Finally, the proof of Lemma \ref{lemma:zonogons2} is completed as follows.
\begin{proof}[Proof of Lemma $\ref{lemma:zonogons2}$]
It is enough to consider the three cases in Lemma \ref{lemma:A2}. Any triangle does not satisfy \eqref{eq:JS_condition2} by the triangle inequality. Secondly, any quadrilateral satisfies \eqref{eq:JS_condition1} by the reverse triangle inequality, and \eqref{eq:JS_condition2} for the quadrilateral means that two pairs of opposite sides are equal in length. Finally, by similar arguments above we can check that the considered hexagon satisfies \eqref{eq:JS_condition1} and \eqref{eq:JS_condition2} if and only if it satisfies \eqref{eq:JS_condition_6-gon}, which is the condition \eqref{eq:JS_condition1} for quadrilateral subdomains of $\overline{\Omega}$ whose consecutive three edges are in $\partial{\Omega}$. 
\end{proof}

\begin{acknowledgement}
The authors express their gratitude to Professor Wayne Rossman for helpful comments.
\end{acknowledgement}

%================================REFERENCES=======================================

%\bibliographystyle{amsxport}  \bibliography{AF2combined} 

% \bib, bibdiv, biblist are defined by the amsrefs package.

\end{document}